\documentclass[11pt]{amsart}
\usepackage{amsmath}
\usepackage{amssymb}
\usepackage{amscd}
\def\NZQ{\mathbb}               % the font for N,Z,Q,R,C
\def\NN{{\NZQ N}}
\def\QQ{{\NZQ Q}}
\def\ZZ{{\NZQ Z}}
\def\RR{{\NZQ R}}

\newtheorem{Theorem}{Theorem}[section]
\newtheorem{Lemma}[Theorem]{Lemma}

\newtheorem{Proposition}[Theorem]{Proposition}

\newtheorem{Definition}[Theorem]{Definition}

\let\epsilon\varepsilon
\let\phi=\varphi
\let\kappa=\varkappa

\def \d {\delta}

\begin{document}

\title{Defect and Local Uniformization}

\author{Steven Dale Cutkosky}
\thanks{Steven Dale Cutkosky was partially supported by NSF}

\address{Steven Dale Cutkosky, Department of Mathematics,
University of Missouri, Columbia, MO 65211, USA}
\email{cutkoskys@missouri.edu}

\author{Hussein Mourtada}
\address{Hussein Mourtada, Equipe G\'eom\'etrie et Dynamique, Institut Math\'ematique de Jussieu-Paris Rive Gauche, Universit\'e Paris 7, B\^atiment Sophie Germain, Case 7012, 75205, Paris Cedex 13, France}
\email{hussein.mourtada.imj-prg.fr}
\thanks{Hussein Mourtada was partially supported by a Miller Fellowship at the University of Missouri}

\dedicatory{Dedicated to Professor Felipe Cano on the occasion of his 60th birthday}
%\keywords{Associated graded ring along a valuation, ramification, finite generation, defect}
%\subjclass{14B05, 14B22, 13B10, 11S15}

\begin{abstract} We give a simple algorithm showing that  the reduction of the multiplicity of a characteristic $p>0$ hypersurface singularity along a valuation is possible if there is a finite linear projection which is defectless. The method begins with the algorithm of Zariski to reduce multiplicity of hypersuface singularities in characteristic 0 along a valuation. This gives a simple demonstration that the only obstruction to local uniformization in positive characteristic is from defect arising in finite projections of singularities. 
\end{abstract}

\maketitle

\section{Introduction}  In this paper, we give a simple algorithm showing that  the reduction of the multiplicity of a characteristic $p>0$ hypersurface singularity along a valuation is possible if there is a finite linear projection which is defectless. The method begins with the algorithm of Zariski to reduce multiplicity of hypersuface singularities in characteristic 0 along a valuation. This gives a simple demonstration that the only obstruction to local uniformization in positive characteristic is from defect arising in finite projections of singularities.

\begin{Definition} Local uniformization holds in dimension $m$ (LU holds in dimension $m$) if for every algebraic function field $K$  over an algebraically closed  field $k$ of dimension $m$ 
and for every valuation $\nu$ of $K/k$, there exists an algebraic local ring $R$ of $K$ such that $R$ is regular and $\nu$ dominates $R$.
\end{Definition}

An algebraic local ring $R$ of a function field $K/k$ is a local domain which is essentially of finite type over $k$ with quotient field $K$. A valuation $\nu$ of an algebraic function field $K/k$
is required to be trivial on $k$. A valuation $\nu$ of $K$ dominates $R$ if $R\subset V_{\nu}$, where $V_{\nu}$ is the valuation ring of $\nu$, and $m_{\nu}\cap R=m_R$ where $m_{\nu}$ is the maximal ideal of $V_{\nu}$ and $m_R$ is the maximal ideal of $R$. We will denote the value group of $\nu$ by $\Phi_{\nu}$. Foundational results on valuations can be found in Chapter VI of \cite{ZS2}, \cite{RTM}, \cite{E} and \cite{R}.
 
 If $X$ is a (necessarily separated) variety with function field $k(X)=K$, then there is at most one point $p\in X$ such that $\nu$ dominates the local ring $\mathcal O_{X,p}$. We will say that $\nu$ dominates $p$, or that the center of $\nu$ on $X$ is $p$. If $X_1\rightarrow X$ is proper and birational and $\nu$ has a center on $X$, then $\nu$ has a center on $X_1$ (by the valuative criterion for properness).

Zariski, \cite{Z2},  found a clever patching argument  (which has been extended to positive characteristic and to other situations by Abhyankar \cite{RES} and Piltant \cite{P}) which proves that local uniformization in dimension $\le 3$ implies resolution of singularities in dimension $\le 3$. However,  there still is not a direct proof (even in characteristic zero) that a set of local uniformizations can be birationally modified  so that they patch together to form a global (proper) resolution of singularities, unless you start out with such a strong local version of resolution of singularities that  patching  becomes unnecessary.

Since $k$ is assumed to be algebraically closed, $K$ is a primitive extension of a rational function field, so we can assume initially in seeking to establish LU that there is a hyper surface singularity whose local ring has $K$ as its quotient field and is dominated by $\nu$.

Local uniformization has been proven in all dimensions over characteristic zero ground fields $k$ by Zariski \cite{Z1} and  local uniformization has been proven in dimension $\le 3$ 
over  ground fields $k$ of  characteristic $p>0$ by Abyankar \cite{RES} and Cossart  and Piltant \cite{CP1} and \cite{CP2}. A reasonably short proof of Abhyankar's result can be found in \cite{C1}.

The existence of resolutions of singularities implies local uniformization, so Hironaka's proof \cite{Hi1} of resolution of singularities in all dimensions over fields of characteristic zero implies local uniformization (in characteristic zero) in all dimensions.

All of the above proofs of local uniformization in dimension $m\ge 3$ (or resolution of singularities in dimension $m\ge 3$) require that embedded local uniformization (or embedded resolution of singularities) be true for hypersurfaces embedded in $\mbox{Spec}(A)$ where $A$ is a polynomial ring in $m$ variables over $k$ (or for hypersurfaces embedded in a  nonsingular variety of dimension $m$).  

Suppose $\nu$ is a valuation whose center is a nonsingular point $p$ on a variety $X$, and $f\in \mathcal O_{X,p}$ is the germ of a hypersurace on $X$ through $p$. Embedded local uniformization holds if there exists a birational morphism $X_1\rightarrow X$ such that $X_1$ is nonsingular at the center $p_1$ of $\nu$ and $f$ is a  unit times a monomial in (suitable) regular parameters in $\mathcal O_{X_1,p_1}$.

There are proofs of resolution of singularities (or local uniformization) after taking a suitable finite extension of $K$ (an alteration). Some of these proofs are by de Jong \cite{dJ} (giving resolution of singularities after a finite extension of  fields), Knaf and Kuhlmann \cite{KK},  Temkin \cite{Te} and  Gabber \cite{I} (giving local uniformization after a finite extension). 
A situation where LU is known in positive characteristic is for Abhyankar valuations, \cite{KK2} and \cite{T2}.

Suppose that $L\rightarrow K$ is a finite field extension and $\omega=\nu|L$. An important invariant of this extension of valued fields is the defect, $\delta(\nu/\omega)$ (Section \ref{Defect}).
If the defect $\delta(\nu/\omega)$ is zero, then the extension can be understood by knowledge of the quotient group $\Phi_{\nu}/\Phi_{\omega}$ and the field extension $V_{\nu}/m_{\nu}$ of $V_{\omega}/m_{\omega}$. The part of the field extension which comes from nontrivial defect is extremely complicated and not well understood. In characteristic zero, the defect is always zero, which is one explanation for why local uniformization is much simpler in characteristic zero.   

The fact that defect is an obstruction to local uniformization in positive characteristic was observed by Kuhlmann in \cite{Ku1}, and other papers such as \cite{Ku3}. The role of defect as an obstruction to local uniformization in characteristic $p>0$ is not readily visible in the proofs of  local uniformization or resolution in characteristic $p>0$ in dimension $\le 3$ (\cite{LU}, \cite{RES}, \cite{Li}, \cite{CJS}, \cite{CP1}, \cite{CP2}, \cite{C1}) and it does not appear in the approaches to resolution of singularities in higher dimension and positive characteristic in \cite{H1}, \cite{BeV} and \cite{BrV}.

In this paper, we show (as follows from Theorem \ref{Theorem4})  that if embedded local uniformization is true  within nonsingular varieties of the dimension of $K$, and a suitable linear projection of a hypersurface singularity with 
function field $K$  which is dominated by a given valuation $\nu$ can always be found  such that if $L$ is the function field of the linear projection and $\omega=\nu|L$ then the defect $\delta(\nu/\omega)=0$, then local uniformization holds in $K$.

A valuation $\nu$ of an algebraic function field $K/k$ is said to be zero dimensional if the transcendence degree of the residue field of the valuation ring of $\nu$ over $k$ is zero. 
This is equivalent to the statement that for every projective variety $X$ with algebraic function field $k$, if $p$ is  the center of $\nu$ on $X$, 
 then $\dim \mathcal O_{X,p}=\dim X$ ($p$ is a closed point of $X$). The essential case of local uniformization is for 0-dimensional valuations (\cite{Z1}, \cite{NS}). For simplicity, we will assume that this condition holds. Further, we will assume that we are in the essential case of a hypersurface singularity  in a polynomial ring. We now state some definitions, within the context which we have just established. 

An extension of domains $R\rightarrow S$ is said to be birational if $R$ and $S$ have the same quotient field.

\begin{Definition}\label{DefELU} Embedded local uniformization holds in dimension $m$ (ELU holds in dimension $m$)  if the following is true: Suppose that $A=k[x_1,\ldots,x_m]$ is a polynomial ring in $m$ variables over an algebraically closed  field $k$ and $\nu$ is a zero dimensional valuation of the quotient field $K$ of $A$ which dominates $A_{\mathfrak m}$ where $\mathfrak m=(x_1,\ldots,x_m)$. Then if $0\ne f\in A$, there exists a birational extension $A\rightarrow A_1$ where $A_1$ is a polynomial ring $A_1=k[x_1(1),\ldots,x_m(1)]$ such that $\nu$ dominates $(A_1)_{\mathfrak m_1}$ where $\mathfrak m_1=(x_1(1),\ldots,x_m(1))$,  there exists $n\le m$ such that $\{\nu(x_1(1)),\ldots,\nu(x_n(1))\}$ is a rational basis of $\Phi_{\nu}\otimes \QQ$ and $f=x_1(1)^{b_1}\cdots x_n(1)^{b_n}\overline f$ with $b_1,\ldots,b_n\in\NN$ and $\overline f\in A_1\setminus \mathfrak m_1$.
\end{Definition}

\begin{Definition} Local reduction of multiplicity holds in dimension $m$ (LRM holds in dimension $m$) if the following is true:
Suppose that $A=k[x_1,\ldots,x_m,x_{m+1}]$ is a polynomial ring in $m+1$ variables over an algebraically closed  field $k$, $f\in A$ is irreducible and  $K$ is the quotient field of 
$A/(f)$, with $f\in \mathfrak m=(x_1,\ldots,x_{m+1})$ and $r=\mbox{ord}(f(0,\ldots,0,x_{m+1}))$ satisfies $1<r<\infty$. Suppose that $\nu$ is a zero dimensional valuation of the quotient field of $A/(f)$ which dominates $(A/(f))_{\mathfrak m}$ . 
Then there exists a birational extension $A\rightarrow A_1$ where $A_1$ is a polynomial ring $A_1=k[x_1(1),\ldots,x_{m+1}(1)]$ such that $\nu$ dominates $(A_1/f_1)_{\mathfrak m_1}$ where $f_1$ is a strict transform of $f$ in $A_1$,
$\mathfrak m_1=(x_1(1),\ldots,x_{m+1}(1))$,   and $1\le r_1=\mbox{ord}(f_1(0,\ldots,0,x_{m+1}(1))<r$.
\end{Definition}

Now ELU in dimension $m$ immediately implies LU in dimension $m-1$ for hypersurfaces, which implies LU in dimension $m-1$ by the primitive element theorem.

Consider the following statements:

\begin{equation}\label{state1} \mbox{LRM in dimension $m$ implies ELU in dimension $m+1$.} 
\end{equation}

\begin{equation}\label{state2} \mbox{ELU in dimension $m$ implies LRM in dimension $m$.}
\end{equation}

If statements (\ref{state1}) and (\ref{state2}) are true, then we could immediately deduce that ELU in dimension $m$ implies ELU in dimension $m+1$, and we would then know that LU holds in all dimensions. 

All of the above cited proofs of local uniformization and resolution of singularities in characteristic zero involve proving the two statements (\ref{state1}) and (\ref{state2}).
The proofs of resolution in dimension three and positive characteristic cited above involve  tricks  to obtain a proof that ELU in dimension 3 and LRM in dimension 3 in the special case  $r=p=\mbox{char k}$ implies LU in dimension 3.  The problem is that we do not  know ELU in dimension 4, so we are unable to proceed to LU in dimension 4.

Now (\ref{state1}) is not so difficult. In fact, induction on $r$ in LRM in dimension $m$ almost gives ELU in dimension $m+1$. So the really hard thing that needs to be proven (to obtain LU)  is (\ref{state2}). Now there are methods, for instance in \cite{Z1} and \cite{NS}, to reduce LRM to the case of rank 1 valuations, so we see that the really essential problem is to prove (\ref{state2}) for rank 1 valuations (the value group $\Phi_{\nu}$ of $\nu$ is order isomorphic to a subgroup of $\RR$).

Zariski's original  characteristic zero proof of (\ref{state2}) for rank 1 valuations from \cite{Z1}  is summarized after Theorem \ref{Theorem3}. The key statement  that fails in positive characteristic  is the validity of equation  (\ref{eq10}) in our summary of the characteristic zero proof, as we have that in characteristic $p>0$,
$$
\binom{r}{r-1}= 0\mbox{ if }p\mbox{ divides }r.
$$
In other words, the problem is the failure of the binomial theorem in positive characteristic, or put more positively, the fact that the Frobenius map is a homomorphism in positive characteristic.

We show in this paper, how starting with Zariski's analysis,  a simple proof can be given of (\ref{state2}) in characteristic $p>0$ for a valuation $\nu$, if the field extension $k(x_1,\ldots,x_m)\rightarrow K$ where $K$ is the quotient field of $k[x_1,\ldots,x_{m+1}]/(f)$ is without defect with respect to $\nu$ and the restriction of $\nu$ to $k(x_1,\ldots,x_m)$. The final statement of this is given in Theorem \ref{Theorem4}.

\section{Asymptotic properties of finite extensions}

Suppose that $K\rightarrow K^*$ is a finite field extension and $R$ is a local domain with maximal ideal $m_R$ and quotient field $K$ and $S$ is a local domain with maximal ideal $m_S$ and quotient field $K^*$. We say that $S$ dominates $R$ if $R\subset S$ and $m_S\cap R=m_R$. We say that $S$ lies over $R$ if $S$ dominates $R$ and $S$ is a localization of the integral closure of $R$ in $K^*$.

An extension of local rings $R\subset R_1$ is said to be a birational extension if $R_1$ dominates $R$, $R_1$ is essentially of finite type over $R$ and $R$ and $R_1$ have the same quotient field.

Suppose that $R$ and $S$ are  normal local rings such that $R$ is excellent, $S$ lies over  $R$, $\nu^*$ is a valuation of the quotient field $K^*$ of $S$ which dominates $S$, and $\nu$ is the restriction of $\nu^*$ to the quotient field $K$ of $R$. Suppose that $K^*$ is finite separable over $K$.  

We let $V_{\nu}$ be the valuation ring of $\nu$, with maximal ideal $m_{\nu}$  and $\Phi_{\nu}$ be the value group of $\nu$. We write 
$e(\nu^*/\nu)=[\Phi_{\nu^*}:\Phi_{\nu}]$ and $f(\nu^*/\nu)=[V_{\nu^*}/m_{\nu^*}:V_{\nu}/m_{\nu}]$.

\begin{Lemma}\label{Lemma4*} (Lemma 5.1 \cite{C3}) Suppose that $S_0$ is a local ring which is a birational extension of $S$ and is dominated by $\nu^*$. Then there exists a normal local ring $R'$ which is a birational extension of $R$ and is dominated by $\nu$, which has the property that if $R''$ is a normal local ring which is a birational extension of $R'$ and is dominated by $\nu$, and if $S''$ is the normal local ring of $K^*$ which lies over $R''$ and is dominated by $\nu^*$, then $S''$ dominates $S_0$.
\end{Lemma}

If $K^*$ is a Galois extension of $K$ with Galois group $G(K^*/K)$ and $A$ is a normal local ring whose quotient field is $K$ and  is dominated by $\nu$ and $B$ is the normal  local ring with quotient field $K^*$ which lies over $A$ and is dominated by $\nu^*$, then the splitting group of $B$ over $A$ is defined as
$$
G^s(B/A)=\{\sigma\in G(K^*/K)\mid \sigma(B)=B\}.
$$
The splitting field of $K^*$ over $K$ is the fixed field 
$$
K^s=(K^*)^{G^s(V_{\nu^*}/V_{\nu})}.
$$

The following lemma follows from Lemma 3.5 \cite{C3}.

\begin{Lemma}\label{Lemma8*} Suppose that $K^*$ is Galois over $K$. Then there exists a birational extension $R'$ of $R$, where $R'$ is a normal local ring which is dominated by $\nu$, such that
if $R''$ is a normal local ring which is a birational extension of $R'$ which is dominated by $\nu$, then 
$$
G^s(C''/R'')=G^s(V_{\nu^*}/V_{\nu})
$$
 where $C''$ is the normal local ring with quotient field $K^*$  which lies over  $R''$ and is dominated by $\nu^*$.
\end{Lemma}

Suppose that $A$ is a subring of $V_{\nu}$. Then the center of $\nu$ on $A$ is the ideal $A\cap m_{\nu}$. An inclusion $A\rightarrow B$ of domains is said to be a birational extension if $A$ and $B$ have the same quotient field and $B$ is essentially of finite type over $A$.

\section{Defect}\label{Defect}
In this section we define the defect of an extension of valued fields, and discuss some basic properties. More details can be found in \cite{ZS2}, \cite{E}, \cite{Ku1} and \cite{R}.

Let $K$ be a field and $\nu$ be a valuation of $K$. Let $K^*$ be a finite Galois extension of $K$ and $\nu^*$ be an extension of $\nu$ to $K^*$. 
Then the defect $\delta(\nu^*/\nu)$ is the natural number defined by the Corollary on page 78 \cite{ZS2} and the discussion on page 58 \cite{ZS2} as
$$
|G^s(V_{\nu^*}/V_{\nu})|=e(\nu^*/\nu)f(\nu^*/\nu)p^{\delta(\nu^*/\nu)}
$$
where
$$
p=\left\{\begin{array}{ll}
\mbox{char }V_{\nu}/m_{\nu}&\mbox{ if }\mbox{char }V_{\nu}/m_{\nu}>0\\
1&\mbox{ if }\mbox{char }V_{\nu}/m_{\nu}=0
\end{array}\right.
$$
and $\delta(\nu^*/\nu)=0$ if $\mbox{char }V_{\nu}/m_{\nu}=0$. 

The definition of the defect $\delta(\nu^*/\nu)$ for an arbitrary finite and separable extension $K^*$ of $K$ is obtained by taking a Galois closure $L$ of $K^*$ over $K$, choosing an extension $\tilde\nu$ of $\nu^*$ to $L$, and defining (formula (26) \cite{CP})
$$
\delta(\nu^*/\nu)=\delta(\tilde\nu/\nu)-\delta(\tilde\nu/\nu^*)\ge 0.
$$
From multiplicativity of the ramification index, residue degree and degree of field extensions, we have that the formula is well defined, and (formula before Definition 7.1 on page 36 of \cite{CP})
\begin{equation}\label{Def}
[(K^*)^s:K^s]=e(\nu^*/\nu)f(\nu^*/\nu)p^{\delta(\nu^*/\nu)}.
\end{equation}
Let $\nu^s$ be the restriction of $\tilde \nu$ to $K^s$ and $(\nu^*)^s$ be the restriction of $\tilde\nu$ to $(K^*)^s$.
Since $e(\tilde \nu/\nu^s)=e(\tilde\nu/\nu)$ and $f(\tilde\nu/\nu^s)=f(\tilde\nu/\nu)$ (Theorem 23, page 71 \cite{ZS2} and Theorem 22, page 70 \cite{ZS2}), we have that $\delta(\tilde\nu/\nu^s)=\delta(\tilde\nu/\nu)$, and 
\begin{equation}\label{eqspdefect}
\begin{array}{lll}
\delta((\nu^*)^s/\nu^s)&=& \delta(\tilde\nu/(\nu^*)^s)-\delta(\tilde\nu/\nu^s)\\
&=& \delta(\tilde\nu/\nu^*)-\delta(\tilde\nu/\nu)\\
&=&\delta(\nu^*/\nu).
\end{array}
\end{equation}

If $K^*$ is a finite extension of $K$ and $\nu^*$ is the unique extension of $\nu$ to $K^*$, then we have 
Ostrowski's lemma (Theorem 2, page 236 \cite{R}), 
\begin{equation}\label{Ost}
[K^*:K]=e(\nu^*/\nu)f(\nu^*/\nu)p^{\delta(\nu^*/\nu)},
\end{equation}
allowing us to  calculate the defect $\delta(\nu^*/\nu)$ of an arbitrary finite extension $K\rightarrow K^*$.

\section{Perron Transforms} Suppose that $K^*$ is a field with a valuation $\nu^*$ and there is a field $k$ which is contained in $V_{\nu^*}$ such that $V_{\nu^*}/m_{\nu^*}=k$ and $\nu^*$ has rank 1.

Suppose that we have a polynomial ring $k[x_1,\ldots,x_m]$ and an irreducible element $f\in k[x_1,\ldots,x_m]$ such that $K^*$ is the quotient field of $k[x_1,\ldots,x_m]/(f)$, and the center of $\nu^*$ on $k[x_1,\ldots,x_m]/(f)$ is the maximal ideal $(\overline x_1,\ldots,\overline x_m)$ where $\overline x_i$ are the residues of $x_i$ in $k[x_1,\ldots,x_m]/(f)$.  
We can extend $\nu^*$ to a pseudo valuation which dominates $k[x_1,\ldots,x_m]$ by prescribing that for $g\in k[x_1,\ldots,x_m]$,
$$
\nu^*(g)=\left\{\begin{array}{ll} \nu^*(\overline g)&\mbox{ if the residue $\overline g$ of $g$ in $k[x_1,\ldots,x_m]/(f)$ is nonzero,}\\
\infty&\mbox{ if }\mbox{$f$ divides $g$ in $k[x_1,\ldots,x_m]$}.
\end{array}\right.
$$
Suppose that the natural map $k[x_1,\ldots,x_{m-1}]\rightarrow k[x_1,\ldots,x_m]/(f)$ is an inclusion and $\overline x_m$ is nonzero. 
Further suppose that $\nu^*(x_1),\ldots,\nu^*(x_n)$ is a rational basis of $\Phi_{\nu^*}\otimes\QQ$ (so that $n\le \mbox{trdeg}_kK^*=m-1$).   

In Section B of \cite{Z1}, Zariski uses the ``Algorithm of Perron'' to define two types of Cremona transforms, (\ref{A6}) and (\ref{A1}) below. 

We construct  a Perron transform
\begin{equation}\label{A6}
x_i=\left\{
\begin{array}{ll}
\prod_{j=1}^nx_j(1)^{a_{ij}}&\mbox{ if  }1\le i\le n\\
x_i(1)&\mbox{ if }n<i\le m
\end{array}\right.
\end{equation}
where $a_{ij}\in \NN$ satisfy $\mbox{Det}(a_{ij})=1$ and $0<\nu^*(x_i(1))<\infty$ for $1\le i\le m$

We now construct another type of Perron transform.
 We have that $\nu^*(x_m)$ is rationally dependent on $\nu^*(x_1),\ldots,\nu^*(x_n)$. There exists a Perron transform 
\begin{equation}\label{A1}
x_i=\left\{\begin{array}{ll}
 (\prod_{j=1}^nx_j(1)^{a_{ij}})(x_m(1)+c)^{a_{i,n+1}}&\mbox{ if }1\le i\le n\\
(\prod_{j=1}^nx_j(1)^{a_{n+1,j}})(x_m(1)+c)^{a_{n+1,n+1}}&\mbox{ if }i=m\\
x_i(1)&\mbox{ if }n<i<m
\end{array}\right.
\end{equation}
with $c\in k$ nonzero, $a_{ij}\in \NN$ and $\mbox{Det}(a_{ij})=1$ such that $0<\nu^*(x_i(1))<\infty$ for $1\le i\le m-1$ and 
$\nu^*(x_1(1)),\ldots,\nu^*(x_n(1))$ is  a rational basis of $\Phi_{\nu^*}\otimes\QQ$. Further, $0<\nu^*(x_m(1))$. We have that $\nu^*(x_m(1))<\infty$, unless $x_m(1)$ is a local equation of the strict transform of $f$ in $k[x_1(1),\ldots,x_m(1)]$.

For Perron transforms of type (\ref{A6}) or (\ref{A1}),
$k[x_1(1),\ldots,x_m(1)]$ is the ring of regular functions of an affine neighborhood of the center of $\nu^*$ on a sequence of blow ups of nonsingular subvarieties above the affine $m$-space with regular functions $k[x_1,\ldots,x_m]$. The maximal ideal of the center of $\nu^*$ on $k[x_1(1),\ldots,x_m(1)]$ is $(x_1(1),\ldots,x_m(1))$. We have an expression
$$
f=x_1(1)^{c_1}\cdots x_m(1)^{c_m}f_1
$$
where $f_1\in k[x_1(1),\ldots,x_m(1)]$ is irreducible. The ring $k[x_1(1),\ldots,x_m(1)]/(f_1)$ is the ring of regular functions of an affine neighborhood of the center of $\nu^*$ on a birational transformation of the  affine variety with regular functions  $k[x_1,\ldots,x_m]/(f)$. 

\begin{Lemma}\label{Lemma11} Suppose that $M_1=x_1^{d_1}\cdots x_n^{d_n}$ and $M_2=x_1^{e_1}\cdots x_n^{e_n}$ with 
$$
d_1,\ldots,d_n,e_1,\ldots,e_n\ge 0
$$
 and
$\nu(M_1)<\nu(M_2)$. Then there exists a Perron transform of type (\ref{A6}) such that $M_1$ divides $M_2$ in $k[x_1(1),\ldots,x_m(1)]$.
\end{Lemma}

This lemma is proved in Theorem 2 \cite{Z1} of Lemma 4.2 \cite{C}.

Suppose that we have a Perron transform of type (\ref{A1}). Let
$A=(a_{ij})^t$, an $(n+1)\times(n+1)$ matrix. Suppose that $d_1,\ldots,d_{n+1},e_1,\ldots,e_{n+1}\in \NN$ and
\begin{equation}\label{A2}
\nu^*(x_1^{d_1}\cdots x_n^{d_n}x_m^{d_{n+1}})=\nu^*(x_1^{e_1}\cdots x_n^{e_n}x_m^{e_{n+1}}).
\end{equation}
We have that
$$
\begin{array}{l}
x_1^{d_1}\cdots x_n^{d_n}x_m^{d_{n+1}}\\
= \left[\prod_{i=1}^n[(\prod_{j=1}^nx_j(1)^{a_{ij}})(x_m(1)+c)^{a_{i,n+1}}]^{d_i}\right]
\left[(\prod_{j=1}^nx_j(1)^{a_{n+1,j}})(x_m(1)+c)^{a_{n+1,n+1}}\right]^{d_{n+1}}\\
=\left(\prod_{j=1}^nx_j(1)^{\sum_{i=1}^{n+1}a_{ij}d_i}\right)(x_m(1)+c)^{\sum_{i=1}^{n+1}a_{i,n+1}d_i}.
\end{array}
$$
Similarly, we have
$$
\begin{array}{l}
x_1^{e_1}\cdots x_n^{e_n}x_m^{e_{n+1}}\\
=\left(\prod_{j=1}^nx_j(1)^{\sum_{i=1}^{n+1}a_{ij}e_i}\right)(x_m(1)+c)^{\sum_{i=1}^{n+1}a_{i,n+1}e_i}.
\end{array}
$$
By (\ref{A2}), we have that 
$$
\sum_{i=1}^{n+1}a_{ij}d_i=\sum_{i=1}^{n+1}a_{ij}e_i\mbox{ for }1\le j\le n.
$$
Let 
$$
\gamma=\sum_{i=1}^{n+1}a_{i,n+1}e_i-\sum_{i=1}^{n+1}a_{i,n+1}d_i.
$$
We have that 
\begin{equation}\label{A5}
A\left(\begin{array}{c}
d_1-e_1\\d_2-e_2\\\vdots\\d_n-e_n\\d_{n+1}-e_{n+1}\end{array}\right)
=\left(\begin{array}{l} 0\\0\\\vdots\\0\\ \gamma\end{array}\right).
\end{equation}
By Cramer's rule, 
\begin{equation}\label{A3}
d_i-e_i=(-1)^{n+1+i}\gamma{\rm Det}(A_{n+1,i})\mbox{ for }1\le i\le n+1
\end{equation}
where $A_{n+1,i}$ is  the submatrix of $A$ obtained by removing the $(n+1)$-st row and $i$-th column from $A$.

\section{Admissible families of valuations and defect} Let $K$ be a field and $\nu$ be a valuation on $K$. Let $K^*=K(z)$ be a finite primitive extension of $K$ with $z\not\in K$ and $\nu^*$ be an extension of $\nu$ to $K^*$. Let $f(x)\in K[x]$ be the minimal polynomial of $z$ over $K$. As explained in \cite{Vaq1} and \cite{Vaq2} (which extend work of MacLane in \cite{Mc1} and \cite{Mc2}), we may represent $\nu^*$ by an admissible family of valuations $\mathcal A=(\mu_{\alpha})_{l\in I}$ of $K[x]$ whose restriction to $K$ is $\nu$, and are successive approximations to $\nu^*$, whose final element is a pseudo valuation which gives $\nu^*$ on $K^*\cong K[x]/(f)$. There is a decomposition (Section 2.1 \cite{Vaq1})
$$
\mathcal A=\mathcal S^{(1)}\cup\cdots \cup \mathcal S^{(N)}
$$
where $N\ge 1$ and each $\mathcal S^{(i)}$ is a simple admissible family (Definition, page 3471 \cite{Vaq1}).  Here $I$ is a totally ordered set, and except for the last element of $I$, which is a pseudo valuation, each $\mu_l$ is either an augmented valuation or a limit augmented valuation. In the first case, of an augmented valuation (Definition on page 3443 \cite{Vaq1}), we have that $\mu_l=[\mu_{l-1};\mu_l(\phi_l)=\gamma_l]$ where $\phi_l$ is a key polynomial (page 3442 \cite{Vaq1}), defining $\mu_l$ in terms of $\mu_{l-1}$, and in the second case, of a limit augmented valuation (Definition on page 3466 \cite{Vaq1}), we have that
$\mu_l=[(\mu_{\alpha})_{\alpha\in A};\mu_l(\phi_l)=\gamma_l]$ where $\phi_l$ is a limit key polynomial (page 3465 \cite{Vaq1}) defining the valuation $\nu_l$ in terms of the continuous family $(\mu_{\alpha})_{\alpha\in A}$.  

Each $S^{(j)}=(\mu_1^{(j)},\ldots,\nu_{n_j}^{(j)};(\mu_{\alpha}^{(j)})_{\alpha\in A^{(j)}})$ for $1\le j\le N$. We have that $\deg \phi_i^{(j)}>\deg \phi_{i-1}^{(j)}$ if $2\le i\le n_j$ and $\deg \phi_{\alpha}^{(j)}=\deg \phi_{n_j}^{(j)}$ for all $\alpha\in A^{(j)}$. Each $(\mu_{\alpha}^{(j)})_{\alpha\in A^{(j)}}$ is a continuous, exhaustive family of iterated augmented valuations (Definition, page 3464 \cite{Vaq1}). In particular, $A^{(j)}$ does not have a maximum.

The first valuation $\mu_1^{(1)}$ of the family $\mathcal S^{(1)}$ is an augmented valuation of the form $\mu_1^{(1)}=[\mu_0;\mu_1^{(1)}(\phi_1^{(1)})=\gamma_1]$ where $\mu_0$ is the valuation $\nu$ of $K$ and $\phi_1^{(1)}$ is a polynomial of degree 1 in $K[x]$. For $t\ge 2$, the first valuation $\mu_1^{(t)}$ of $\mathcal S^{(t)}$ is a limit augmented valuation for the family $(\mu_{\alpha}^{(t-1)})_{\alpha\in A^{(t-1)}}$. $\mathcal S^{(N)}$ has a largest element, which is either an augmented valuation, or a limit augmented valuation $\mu_{\overline l}$ associated to the polynomial $\phi_{\overline l}=f$ with value $\gamma_{\overline l}=\infty$.

For $1\le j\le N-1$, we have associated rational numbers $S^{(j)}(\mathcal A)$ called   jumps,   and the total jump of the extension is (page 861 \cite{Vaq2})
$$
s^{\rm tot}(\mathcal A)=\prod_{j=2}^Ns^{(j-1)}(\mathcal A).
$$
The jump $S^{(j-1)}(\mathcal A)$ is defined (Definition 2.7, page 870 \cite{Vaq2}) by 
$$
\deg \phi_1^{(j)}=S^{(j-1)}(\mathcal A) \deg \phi_{\alpha}^{(j-1)}
$$
where $\phi_1^{(j)}$ is the limit key polynomial defining the valuation $\mu_1^{(j)}$ and $\phi_{\alpha}^{(j-1)}$ is a key polynomial associated to the continuous family $(\mu_{\alpha}^{(j-1)})_{\alpha\in A^{(j-1)}}$.

By Lemma 2.11 \cite{Vaq2}
\begin{equation}\label{Jump1}
S^{(j-1)}(\mathcal A)>1\mbox{ for }2\le j\le N
\end{equation}
and by Corollary 2.10 \cite{Vaq2},
\begin{equation}\label{Jump2}
[K^*:K]=e(\nu^*/\nu)f(\nu^*/\nu)S^{\rm tot}(\mathcal A).
\end{equation}
Thus if $\nu^*$ is the unique extension of $\nu$ to $K^*$, then by Ostrowski's lemma, equation (\ref{Ost}),
\begin{equation}\label{jumpdefect}
p^\delta(\nu^*/\nu)=\mathcal S^{\rm tot}(\mathcal A).
\end{equation}

\begin{Proposition}\label{Prop1} Suppose that 
$$
\nu^*(z-K)=\{\nu^*(z-a)\mid a\in K\}
$$
 does not have a largest element. Then $N\ge 2$.
\end{Proposition}

The set $\nu^*(z-K)$  determines  the distance of $z$ from $K$, investigated by Kuhlmann in \cite{Ku3}. 

\begin{proof} Since $\nu^*(z-K)$ does not have a largest element, we calculate from  Proposition 2.3 and Theorem 2.4 \cite{Vaq1} that 
$$
S^{(1)}=\mathcal D^{(1)}\cup \mathcal C^{(1)}
$$
where $\mathcal D^{(1)}=(\nu_1^{(1)})$ and $\mathcal C^{(1)}=(\nu_{\alpha}^{(1)})_{\alpha\in A}$, with 
$\nu_1^{(1)}=[\nu;\nu_1^{(1)}(\phi_1^{(1)})=\gamma_1^{(1)}]$ with $\phi_1^{(1)}=z$ and $\gamma_1^{(1)}=\nu^*(z)$ and 
$$
A=\{\alpha\in \nu^*(z-K)\mid \alpha>\gamma_1^{(1)}\},
$$
and for $\alpha\in A$, 
$$
\nu_{\alpha}^{(1)}=[\nu_1^{(1)};\nu_{\alpha}^{(1)}(\phi_{\alpha}^{(1)})=\alpha],
$$
where $\phi_{\alpha}^{(1)}=z-h_{\alpha}$ for some $h_{\alpha}\in K$ such that $\nu^*(\phi_{\alpha}^{(1)})=\alpha$. We have that $S^{(2)}$ begins with the augmented limit valuation 
$$
\nu_1^{(2)}=\lim_{\alpha\in A}\nu_{\alpha}^{(1)}.
$$
Thus $N\ge 2$.
\end{proof}

\begin{Lemma}\label{Lemma2} Suppose that $K$ contains a field $k$ such that $k\subset V_{\nu^*}$ and $V_{\nu^*}/m_{\nu^*}=k$. Further suppose that $\nu^*(z-K)$ has a largest element $\gamma$. Then $\gamma\not\in \Phi_{\nu}$.
\end{Lemma}

\begin{proof} Suppose that $\gamma\in \Phi_{\nu}$.  Let $h\in K$ be such that $\nu^*(z-h)=\gamma$. Since $\gamma\in \Phi_{\nu}$, there exists $g\in K$ such that $\nu(g)=\gamma$. Let $c$ be the class of $\frac{z-h}{g}$ in $V_{\nu^*}/m_{\nu^*}=k$. Then $\nu^*(z-h-cg)>\nu^*(z-h)=\gamma$, a contradiction.
\end{proof}

\section{An algorithm for reduction of multiplicity}\label{Alg}

Let $k$ be a field, $T$ be a polynomial  ring $T=k[x_1,\ldots,x_m]$, $f\in T$ be irreducible and monic in $x_m$ and $S=T/(f)$. Let $K^*$ be the quotient field of $S$. Let $z$ be the class of $x_m$ in $S\subset K^*$. Assume that $K^*$ is separable over $K$. Let $R=k[x_1,\ldots,x_{m-1}]$ and $K$ be the quotient field of $R$, so that $K^*$ is a finite extension of $K$. Let $r=\mbox{ord }f(0,\ldots,0,x_m)$. Let $\nu^*$ be a valuation of $K^*$ with restriction $\nu$ to $K$. Suppose that  $\nu^*$  has rank 1 and $V_{\nu^*}/m_{\nu^*}=k$. Suppose that $S\subset V_{\nu^*}$.

We can regard $\nu^*$ as a pseudo valuation on $T$, where for $g\in T$, $\nu^*(g)=\infty$ 
if $f$ divides $g$.
Suppose that the center of $\nu^*$ on $T$ is $(x_1,\ldots,x_m)$.

The first part of the algorithm of Theorem \ref{Theorem3}  (through equation (\ref{A30})) is the same as in \cite{Z1}.

%Suppose that $K$ is an algebraic function field over a field $k$,  $\nu$ has rank 1 (so that $\nu^*$ also has rank 1) and $V_{\nu^*}/m_{\nu^*}=k$.

%Suppose that  $\nu^*(z)>0$ and there exists an algebraic  regular local ring $R$ of $K$ which is dominated by $\nu$  such that the coefficients of $f(x)$ are in $R$, so that $R[z]\cong R[x]/(f(x))$. 

%Let $S=R[z]_{m_{\nu^*}\cap R[z]}\cong (R[x]/(f(x))_{m_R+(x)}$. We will identify $\nu^*$ with an extension of $\nu^*$ to a rank 2 valuation of the field $K(x)$. We will write $\nu^*(f)=\infty$.  We will suppose that embedded resolution of singularities is true within regular varieties of dimension equal to the transcendence degree of $K$ over $k$.

%\begin{Theorem}\label{Theorem3} Let  assumptions be as above, and suppose  that embedded resolution of singularities is true within regular varieties of dimension equal to the transcendence degree of $K$ over $k$. Let  $r=\mbox{ord}_{R/m_R[x]}(f)$ and suppose that $r>1$ and $\nu^*(z)\not\in \Phi_{\nu}$. Then there exists a sequence of 
% transforms  
%$S\rightarrow S'$ along $\nu^*$ such that $S'\cong R'[x']/(f'(x')))_{m_{R'}+(x')}$ where $R'$ is a regular local ring and $\mbox{ord}_{R'/m_{R'}}(f'(x'))<r$.
%\end{Theorem}

\begin{Theorem}\label{Theorem3} Let  assumptions be as above, and suppose  that embedded local uniformization (Definition \ref{DefELU})) is true in dimension  $m-1$. Suppose that $r>1$ and $\nu^*(z)\not\in \Phi_{\nu}$. Then there exists a birational extension along $\overline\nu^*$,
$T\rightarrow T'=k[x_1',\ldots,x_m']$, such that $T'\subset V_{\overline \nu^*}$ and $(x_1',\ldots,x_m')$ is the center of $\overline\nu^*$ on $T'$, and if $f'$ is the strict transform of $f$ in $T'$, so that $S'=T'/(f')$ is a birational extension $S\rightarrow S'$ along $\nu^*$, we have that $\mbox{ord }f'(0,\ldots,0,x_m')<r$.
\end{Theorem}

\begin{proof} Let $n$ be the rank of $\nu^*$.  Since $K^*$ is a finite extension of $K$, we have that $\mbox{rank }\nu=1$ and $\mbox{rat rank }\nu^*=\mbox{rat rank }\nu=n$. By embedded local uniformization in $R$, there exists a birational transform $R\rightarrow R_1=k[x_1(1),\ldots,x_{m-1}(1)]$ such that the center of $\nu$ on $R_1$ is $(x_1(1),\ldots,x_{m-1}(1))$  and $\nu(x_1(1)),\ldots,\nu(x_n(1))$ is a rational basis of $\Phi_{\nu}\otimes \QQ$. Let $x_m(1)=x_m$ and $T_1$ be the polynomial ring $T_1=k[x_1(1),\ldots, x_m(1)]$. We have that the center of $\nu^*$ on $T_1$ is  $(x_1(1),\ldots,x_m(1))$. Since $f$ is monic in $x_m$, $f$ is irreducible in $T_1$, and is monic in $x_m(1)$. Thus $S_1=T_1/(f)$ is  a domain with quotient field $K^*$, so $S\rightarrow S_1$ is a birational extension.

Expand
\begin{equation}\label{A22}
f=x_m(1)^e+a_{e-1}x_m(1)^{e-1}+\cdots+a_0
\end{equation}
with $a_i\in R_1$ for $0\le i \le e-1$. Let $a_e=1$. By embedded local uniformization in $R_1$, there exists a birational extension $R_1\rightarrow R_2=k[x_1(1),\ldots,x_{m-1}]$ along $\nu$ such that the center of $\nu$ on $R_2$ is $(x_1(2),\ldots,x_{m-1}(2))$, 
$$
\nu(x_1(2)),\ldots, \nu(x_n(2))
$$
 is a rational basis of $\Phi_{\nu}\otimes \QQ$ and if $a_i\ne 0$, then 
$$
a_i=x_1(2)^{d_1(i)}\cdots x_n(2)^{d_n(i)}\overline a_i
$$
with $d_1(i),\ldots,d_n(i)\in \NN$ and $\overline a_i\not\in (x_1(2),\ldots,x_{m-1}(2))$. Let $x_m(2)=x_m(1)$ and $T_2=k[x_1(2),\ldots,x_m(2)]$. We have that $d_1(e)=\cdots=d_n(e)=0$.
Let 
$$
\rho=\min\{\nu^*(a_ix_m(2)^i)\mid 0\le i\le e\}.
$$
There exist $t$ natural numbers $\sigma_i$ with
$$
0\le \sigma_1<\sigma_2<\cdots<\sigma_t\le r
$$
such that $\nu^*(a_{\sigma_i}x_m(2)^{\sigma_i})=\rho$ for $1\le i\le t$ and $\nu^*(a_lx_m(2)^l)>\rho$ if $l\ne \sigma_i$ for some $i$ with $1\le i\le t$. Since $\nu^*(f)=\infty$, we have that
\begin{equation}\label{A4}
1<t.
\end{equation}
We now perform a Perron transform $T_2\rightarrow T_3=k[x_1(3),\ldots,x_m(3)]$ of type (\ref{A1}) along $\nu^*$,
$$
x_i(2)=\left\{\begin{array}{ll}
 (\prod_{j=1}^nx_j(3)^{a_{ij}})(x_m(3)+c)^{a_{i,n+1}}&\mbox{ if }1\le i\le n\\
(\prod_{j=1}^nx_j(3)^{a_{n+1,j}})(x_m(3)+c)^{a_{n+1,n+1}}&\mbox{ if }i=m\\
x_i(3)&\mbox{ if }n<i<m.
\end{array}\right.
$$
Let $A=(a_{ij})^t$.

For $1\le l\le r$ and $1\le j\le n$, let 
$$
\tau_{j,l}=\left(\sum_{i=1}^na_{ij}d_i(l)\right)+(a_{n+1,j})l
$$
and for $1\le l\le r$, let
$$
\lambda_l=\sum_{i=1}^n a_{i,n+1}d_i(l)+(a_{n+1,n+1})l.
$$
Then 
$$
a_lx_m(1)^l=\overline a_l\left(\prod_{j=1}^nx_j(3)^{\tau_{j,l}}\right)(x_m(3)+c)^{\lambda_l}
$$
for $0\le l\le e$ and if $a_l\ne 0$. We have
$$
\tau_{j,\sigma_i}=\tau_{j,\sigma_1}
$$
for $1\le j\le n$ and $1\le i\le t$ since $\nu(x_1(3)),\ldots,\nu(x_n(3))$ is a rational basis of $\Phi_{\nu}\otimes\QQ$. By Lemma \ref{Lemma11}, after possibly performing a Perron transform of type (\ref{A6})  in $x_1(3),\ldots,x_n(3)$,
we have an expression
$$
f=\prod_{j=1}^nx_j(3)^{\tau_{j,\sigma_1}}f_1
$$
where
\begin{equation}\label{A20}
f_1=\sum_{i=1}^t\overline a_{\sigma_i}(x_m(3)+c)^{\lambda_{\sigma_i}}+h
\end{equation}
with $h\in (x_1(3),\ldots,x_n(3))T_3$ is a local equation of the strict transform of $f$ in $T_3$. Let $d=\mbox{Det}(A_{n+1,n+1})$. By (\ref{A3}),
\begin{equation}\label{A7}
(\lambda_{\sigma_i}-\lambda_{\sigma_1})d=\sigma_i-\sigma_1\mbox{ for }1\le i\le t.
\end{equation}
By (\ref{A4}), $t>1$ so $d\ne 0$. Suppose $d>0$. Then
$$
\begin{array}{lll}
f_1&=& (x_m(3)+c)^{\lambda_{\sigma_1}}\left(\sum_{i=1}^t\overline a_{\sigma_i}(x_m(3)+c)^{\lambda_{\sigma_i}-\lambda_{\sigma_1}}+h\right)\\
&=& (x_m(3)+c)^{\lambda_{\sigma_1}}\left(\sum_{i=1}^t\overline a_{\sigma_i}(x_m(3)+c)^{\frac{\sigma_i-\sigma_1}{d}}+h\right).
\end{array}
$$
Let $r_1=\mbox{ord}(f_1(0,\ldots,0,x_m(3))$. We have that
$$
r_1\le \frac{\sigma_t-\sigma_1}{d}\le r.
$$
Suppose that $r_1=r$. Then 
\begin{equation}\label{A30}
\sigma_t=r, \sigma_1=0\mbox{ and }d=1.
\end{equation}

So far, the proof has been as Zariski's in \cite{Z1}. The remainder of his proof requires characteristic zero, and is not valid in characteristic $p>0$. We provide a different analysis from here on. After this proof, we give an outline of  the conclusion of Zariski's characteristic zero proof.

 By (\ref{A5}),
$$
A\left(\begin{array}{c}
d_1(\sigma_1)-d_1(\sigma_t)\\
\vdots\\
d_n(\sigma_1)-d_n(\sigma_t)\\
\sigma_1-\sigma_t
\end{array}\right)=
A\left(\begin{array}{c}
d_1(\sigma_1)\\\vdots\\ d_n(\sigma_1)\\ -r\end{array}\right)
=\left(\begin{array}{c} 0\\ \vdots\\ 0\\ \gamma\end{array}\right)
$$
where $\gamma=\lambda_{\sigma_1}-\lambda_{\sigma_t}$.  By (\ref{A7}) and since $d=1$,
$$
\gamma=-r.
$$
Further, by (\ref{A3}), $d_i(\sigma_1)=(-1)^{n+i}r\mbox{Det}(A_{n+1,i})$ for $1\le i\le n$. Thus $r$ divides $d_i(\sigma_1)$ for $1\le i\le n$. Since 
$$
r\nu^*(x_m)=\nu^*(x_m^r)=\nu^*(x_m(2)^r)=\nu(x_1(2)^{d_1(\sigma_1)}\cdots x_n(2)^{d_n(\sigma_1)})=\sum_{i=1}^nd_i(\sigma_1)\nu(x_i(2)),
$$
we have that $\nu^*(x_m)\in \Phi_{\nu}$.

If $d<0$, we have a similar argument, writing
$$
f_1=(x_m(3)+c)^{\lambda_{\sigma_t}}(\sum_{i=1}^t\overline a_{\sigma_i}(x_m(3)+c)^{\frac{\sigma_i-\sigma_t}{d}}+h).
$$

\end{proof}

Zariski's proof of reduction of singularities in characteristic zero in \cite{Z1} (or Theorem 8.4 \cite{RES}), proceeds after (\ref{A30}) as follows. If we have $r_1=r$ in the above algorithm, we have that
$$
\sum_{i=1}^t\overline a_{\sigma_i}(0,\ldots,0)(x_m(3)+c)^{\sigma_i}=\overline a_r(0,\ldots,0)x_m(3)^r
$$
 in $k[x_m(3)]$.  Set 
 $$
 g(u)=\sum_{i=1}^t\overline a_{\sigma_i}(0,\ldots,0)u^{\sigma_i}=\overline a_r(0,\ldots,0)(u-c)^r.
 $$
The binomial theorem tells us (since we are in characteristic zero) that 
\begin{equation}\label{eq10}
\sigma_{t-1}=r-1
\end{equation}
 so $\nu^*(x_m(1))=\nu^*(a_{r-1})$. There thus exists $\omega\in k=V_{\nu^*}/m_{\nu^*}$ such that 
$$
\nu^*(x_m(1)-\omega a_{r-1})>\nu^*(x_m(1)).
$$
We now make a change of variables in $T_1$, replacing $x_m(1)$ with $x_m(1)'=x_m(1)-\omega a_{r-1}$. After a finite number of iterations of the algorithm from (\ref{A22}), we must obtain a reduction $r_1<r$, since otherwise we construct an infinite increasing sequence of values in  $R_1$ which is bounded above by $\nu^*\left(\frac{\partial f}{\partial x_m(1)}\right)$, which is impossible since $R_1$ is Noetherian,

\section{Reduction of multiplicity under a defectless projection}

%\begin{Theorem}\label{Theorem4} Let assumptions be as in Section \ref{Alg} and suppose that embedded resolution of singularities is true within regular varieties of dimension equal to the transcendence degree of $K$ over $k$ and the defect $\delta(\nu^*/\nu)=0$. Let  $r=\mbox{ord}_{R/m_R[x]}(f)$ and suppose that $r>1$
%Then there exists a sequence of monoidal transforms  
%$S\rightarrow S'$ along $\nu^*$ such that $S'\cong R'[x']/(f'(x')))_{m_{R'}+(x')}$ where $R'$ is a regular local ring and $\mbox{ord}_{R'/m_{R'}}(f'(x'))<r$
%\end{Theorem}

\begin{Theorem}\label{Theorem4} Let assumptions be as in Section \ref{Alg} and suppose that embedded local uniformization (Definition \ref{DefELU}) is true in dimension $m-1$,  
$$
r=\mbox{ord }f(0,\ldots,x_m)>1
$$
 and the defect $\delta(\nu^*/\nu)=0$.
Then there exists a birational transform along $\nu^*$,
$T\rightarrow T'=k[x_1',\ldots,x_m']$, such that $(x_1',\ldots,x_m')$ is the center of $\nu^*$ on $T'$, and if $f'$ is the strict transform of $f$ in $T'$, so that $S'=T'/(f')$ is a birational extension $S\rightarrow S'$ along $\nu^*$, we have that $\mbox{ord }f'(0,\ldots,0,x_m')<r$.
\end{Theorem}

\begin{proof} Let $L$ be a Galois closure of $K^*$ over $K$ and $\tilde\nu$ be an extension of $\nu^*$ to $L$. With notation as in formula (\ref{eqspdefect}), we have that 
$\delta((\nu^*)^s/\nu^s)=\delta(\nu^*/\nu)=0$. Further, $(\nu^*)^s$ is the unique extension of $\nu^s$ to $(K^*)^s$ by Proposition 1.46 \cite{RTM}. We have that $\Phi_{\nu^s}=\Phi_{\nu}$, $\Phi_{(\nu^*)^s}=\Phi_{\nu^*}$  by Theorem 23, page 71 \cite{ZS2} and 
$V_{\nu^s}/m_{\nu^s}=V_{\nu}/m_{\nu}=k$, $V_{(\nu^*)^s}/m_{(\nu^*)^s}=V_{\nu^*}/m_{\nu^*}=k$ by Theorem 22, page 70 \cite{ZS2}, so by(\ref{Def}) or  (\ref{Ost}),
$$
[(K^*)^s:K^s]=e((\nu^*)^s/\nu^s)=e(\nu^*/\nu).
$$
First suppose that $(K^*)^s\ne K^s$. Recall that  $z$ is the class of $x_m$ in $S\subset K^*$.
Since $(K^*)^s=K^*K^s=K^s(z)$, 
 we have that $(\nu^*)^s(z-K^s)$ has a largest element $\gamma\in \Phi_{(\nu^*)^s}=\Phi_{\nu^*}$ by (\ref{Jump1}), (\ref{jumpdefect})  and Proposition \ref{Prop1}. Further, 
 \begin{equation}\label{A12} 
 \gamma\not\in \Phi_{\nu^s}=\Phi_{\nu}
 \end{equation}
  by Lemma \ref{Lemma2}. Let $h\in K^s$ be such that $(\nu^*)^s(z-h)=\gamma$. We have that $\gamma\ge \nu^*(z)>0$ and $\nu^s(h)\ge \nu^*(z)>0$. 
 
 Let $U$ be the localization of the integral closure of $R$ in $K^s$ at the center of $\nu^s$. There exists a birational extension $U\rightarrow U_0$ where $U_0$ is a normal local ring which is dominated by $\nu^s$ such that $h\in U_0$ (just take $U_0$ to be the localization of the integral closure of $U[h]$ at the center of $\nu^s$). By Lemma \ref{Lemma8*} and Lemma \ref{Lemma4*} and by embedded local uniformization, there exists a birational extension $R\rightarrow R_1=k[x_1(1),\ldots,x_{m-1}(1)]$ along $\nu$ such that the center of $\nu$ on $R_1$ is $\mathfrak m_1=(x_1(1),\ldots,x_{m-1}(1))$ and 
 \begin{equation}\label{A11}
 G^s(W_1/(R_1)_{\mathfrak m_1})=G^s(V_{\tilde\nu}/V_{\nu})
 \end{equation}
 where $W_1$ is the localization of the integral closure of $R_1$ in $L$ at the center of $\tilde\nu$ and $h\in U_1$ where $U_1$ is the localization of the integral closure of $R_1$ in $K^s$ at the center of $\nu^s$. By (\ref{A11}) and Theorem 1.47 \cite{RTM}, we have that $\mathfrak m_1 U_1$ is the maximal ideal of $U_1$ and the residue field of $U_1$ is $k=R_1/\mathfrak m_1$, so $U_1$ is a regular local ring with completion $\hat U_1=\hat R_1$, where the completion of $R_1$ is at the maximal ideal $\mathfrak m_1$. Thus $h\in \hat R_1$.

Now $\gamma<\nu((\mathfrak m_1 \hat R_1)^t)$ for some $t\in \ZZ_+$ since $\nu$ has rank 1, and thus there exists $h'\in R_1$ such that $h-h'\in (\mathfrak m_1\hat R_1)^t$. Thus
$$
\nu^*(z-h')=(\nu^*)^s((z-h)+(h-h'))=\nu^*(z-h)=\gamma.
$$
We have that $\gamma$ is the largest element of $\nu^*(z-K)$ since $\nu^*(z-K)\subset (\nu^*)^s(z-K^s)$. Thus $\nu^*(z-h')=\gamma\not\in \Phi_{\nu}$ by (\ref{A12}).

Let $x_m(1)=x_m$ and $T_1=k[x_1(1),\ldots,x_m(1)]$ so that the center of $\tilde\nu$ on $T_1$ is 
$$
(x_1(1),\ldots,x_m(1)).
$$
Now $R_1[z]\cong T_1/(f)$, and regarding $f$ as a polynomial in $T_1$, we have that  
$$
\mbox{ord }f(0,\ldots,0,x_m(1))=r.
$$
 Making  the change of variable in $T_1$ replacing $x_m(1)$ with $x_m(1)-h'$ and $z$ with $z-h'$, we have that $\nu^*(z)\not\in \Phi_{\nu}$. Then the reduction of $r$ in the conclusions of the theorem follows from Theorem \ref{Theorem3}.

Now suppose that $K^s=(K^*)^s$. Then $z\in K^s$. As in the above case, there exists a birational transformation $R\rightarrow R_1$ along $\nu$ such that $R_1=k[x_1(1),\ldots,x_{m-1}(1)]$, the center of $\nu$ on $R_1$ is $m_{R_1}=(x_1(1),\ldots,x_{m-1}(1))$  and
$z\in U_1$, where $U_1$ is the normal local ring of $K^s$ which is dominated by $\nu^s$, and  $U_1$ is a regular local ring with $\hat U_1=\hat R_1$ where the completion of $R_1$ is with respect to the maximal ideal $\mathfrak m_1=(x_1(1),\ldots,x_{m-1}(1))$ of $R_1$.
Letting $x=x_m$ and writing $f(x)=f(x_1(1),\ldots,x_{m-1}(1),x)\in U_1[x]$, we have a factorization $f(x)=(x-z)f_1(x)$ with $f_1(x)\in U_1[x]$.  Since $K^*$ is separable over $K$, we  have that $f_1(z)\in K^s$ is nonzero, and has nonnegative  value. Let $\tau=\nu^s(f_1(z))\ge 0$.
Since $\Phi_{\nu^*}=\Phi_{\nu}$ (by Theorem 23, page 71 \cite{ZS2}) and $\nu$ has rank 1,  we may
replace $x$ with $x-h$ and $z$ with $z-h$ for suitable $h\in R_1$ so that $\nu^s(z)>\tau$. Then writing 
\begin{equation}\label{A10}
f_1(x)=a_dx^d+a_{d-1}x^{d-1}+\cdots+a_0
\end{equation}
with $a_i\in U_1$ and $d\in \NN$, we have that $\nu^s(z^ia_i)>\tau$ for $i>0$ and  $\nu^s(a_0)=\tau$. 
Let $G=G(L/K)$ be the Galois group of $L$ over $K$ and for $0\le i\le d-1$ such that $a_i\ne 0$, let $b_i=\prod_{\sigma\in G}\sigma(a_i)\in R_1$ and let $c=\prod_{\sigma\in G}\sigma(z)\in R_1$.

 By embedded local uniformization  in $R_1$, applied to $c\prod b_i$, there is a birational extension $R_1\rightarrow R_2$ along $\nu$ such that $R_2=k[x_1(2),\ldots,x_{m-1}(2)]$, the center of $\nu$ on $R_2$ is $\mathfrak m_2=(x_1(2),\ldots,x_{m-1}(2))$ and  
each $b_i$ and $c$ is a monomial in $x_1(2),\ldots,x_n(2)$ (where $\nu(x_1(2)),\ldots \nu(x_n(2))$ is a rational basis of $\Phi_{\nu}\otimes \QQ$) times a unit in 
$(R_2)_{\mathfrak m_2}$. Let
 $U_2$ be the normal local ring of $K^s$ which is dominated by $\nu^s$ and lies over $(R_2)_{\mathfrak m_2}$. Then $U_2$ is a regular local ring with $\hat U_2=\hat R_2$ (by Lemma \ref{Lemma8*}). In particular, $x_1(2),\ldots,x_{m-1}(2)$ are regular parameters in $U_2$ and each nonzero 
 \begin{equation}\label{eqB1}
 a_i=x_1(2)^{c_1(i)}\cdots x_n(2)^{c_n(i)}\phi_i
 \end{equation}
 is a monomial in $x_1(1),\ldots, x_n(2)$ times a unit $\phi_i$ in $U_2$ and $z=x_1(2)^{b_1}\cdots x_n(2)^{b_n}\lambda$ is a monomial in $x_1(2),\ldots,x_n(2)$ times a unit $\lambda$ in $U_2$.

Let  $T_2=k[x_1(2),\ldots,x_{m-1}(2),x_m]$, a birational extension of $T$, such that the center of $\nu^*$ on $T_2$ is $(x_1(2),\ldots,x_{m-1}(2),x_m)$. Now perform the Perron transform $T_2\rightarrow T_3$ along $\nu^*$ defined by $x_i(2)=x_i(3)$ for $1\le i\le m-1$ and 
\begin{equation}\label{A9}
x_m=x_1(3)^{b_1}\cdots x_n(3)^{b_n}(x_m(3)+\beta)
\end{equation}
where $\beta$ is the residue of $\lambda$ in $V_{\nu^s}/m_{\nu^s}\cong k$. Thus
$$
x_m-z=x_1(3)^{b_1}\cdots x_n(3)^{b_n}(x_m(3)-(\lambda-\beta))
$$
in $U_2[x_m(3)]$. 
After substituting $x=x_m$, (\ref{eqB1}) and (\ref{A9}) into (\ref{A10}) and possibly performing a Perron transform $R_2\rightarrow R_4=k[x_1(4),\ldots,x_{m-1}(4)]$ along $\nu$ of type (\ref{A6}), and letting $U_4$ be the normal local ring of $K^s$ which lies over $(R_4)_{\mathfrak m_4}$ and is dominated by $\nu^s$, where  $\mathfrak m_4=(x_1(4),\ldots,x_{m-1}(4))$,
so that $U_4$ is a regular local ring with $\hat U_4=\hat R_4$ where the completion of $R_4$ is with respect to the maximal ideal $\mathfrak m_4$ of $R_4$ (by Lemma \ref{Lemma8*}), we have that the strict transform $\overline f_1$ of $f_1$ in $U_4[x_m(3)]$ is $\overline f_1= e_0+e_1 x_m(3)+\cdots+e_dx_m(3)^d$
with $e_i\in U_4$ for $1\le i$ and $e_0$ is a unit in $U_4$. Letting $\overline f$ be the strict transform of $f(x_m)$ in $T_4=R_4[x_m(3)]$, we have that $\overline f=(x_m(3)-(\lambda-\beta))\overline f_1$ in $U_4[x_m(3)]$, so that $\mbox{ord }\overline f(0,\ldots,0,x_m(3))=1$, since $\lambda-\beta\in \mathfrak m_4U_4$.

\end{proof}

\end{document}